\documentclass[12pt]{amsart}
\usepackage[margin=1in]{geometry}
\usepackage{amssymb,amsfonts,latexsym,amsmath,amsthm,graphicx}
\usepackage{parskip}
\usepackage{mathrsfs}
\usepackage{mathtools} 
\mathtoolsset{showonlyrefs,showmanualtags}

\usepackage[margin=1in]{geometry}

\numberwithin{equation}{section}
\newtheorem{thm}{Theorem}[section]

\newtheorem{lemma}[thm]{Lemma}

\theoremstyle{remark} 
\newtheorem{remark}[]{Remark}

\newcommand{\EE}{\mathbb{E}}
\newcommand{\PP}{\mathbb{P}}
\newcommand{\RR}{\mathbb{R}}
\newcommand{\CC}{\mathbb{C}}

\usepackage[colorinlistoftodos]{todonotes}

\newcommand{\ol}{\overline}

\newcommand{\e}{\varepsilon}

\newcommand{\be}{\begin{equation}}
\newcommand{\ee}{\end{equation}}

\begin{document}

\title[Valence of harmonic polynomials]{The valence of harmonic polynomials viewed through the probabilistic lens}
\date{}

\author[E. Lundberg]{Erik Lundberg}
\address{Department of Mathematical Sciences,
Florida Atlantic University, Boca Raton, FL 33431}
\email{elundber@fau.edu}
%    \thanks will become a 1st page footnote.
\thanks{The author acknowledges support from the Simons Foundation (grant 712397).}

\begin{abstract}
We prove the existence of complex polynomials $p(z)$ of degree $n$ and $q(z)$ of degree $m<n$ such that the harmonic polynomial $ p(z) + \overline{q(z)}$ has at least $\lceil n \sqrt{m} \rceil$ many zeros.  
This provides an array of new counterexamples to Wilmshurst's conjecture that the maximum valence of harmonic polynomials $p(z)+\overline{q(z)}$ taken over polynomials $p$ of degree $n$ and $q$ of degree $m$ is $m(m-1)+3n-2$.  More broadly, these examples show that there does not exist a linear (in $n$) bound on the valence with a uniform (in $m$) growth rate.  The proof of this result uses a probabilistic technique based on estimating the average number of zeros of a certain family of random harmonic polynomials.
\end{abstract}

\maketitle

\section{Introduction}

T. Sheil-Small \cite{SS} posed the problem of determining the maximum valence (maximum number of preimages of a given point) of complex-valued harmonic polynomials $h(z)=p(z) + \ol{q(z)}$ with $n:=\deg p > m := \deg q > 0$. He conjectured that the maximum is $n^2$.
In his thesis \cite{W1}, A. S. Wilmshurst used the maximum principle for harmonic functions to show that the valence of $p(z) + \ol{q(z)}$ is finite, and he then applied Bezout's theorem to confirm Sheil-Small's conjecture that the valence is at most $n^2$.  He also provided examples with $m=n-1$ where this upper bound is attained.

Wilmshurst conjectured the following improvement on this upper bound for each $m<n-1$. For each pair of integers $n>m>1$, let $N(n,m) := \max | \{z\in \CC : p(z) + \ol{q(z)} = 0\} |$ denote the maximum valence, where the maximum is taken over polynomials $p$ of degree $n$ and $q$ of degree $m$.
Then Wilmshurst's conjecture \cite{W1} (cf. \cite{W2}) for the maximum valence is
\begin{equation}\label{eq:Wilmshurst}
N(n,m) = 3n-2 + m(m-1).
\end{equation}

This 1994 conjecture was stated in \cite[Remark 2]{W2} and was discussed further in \cite{SS}.
It is also mentioned in the list of open problems in \cite{BL2010}.
For $m=n-1$ the conjecture states that the maximum valence is $N(n,n-1)=n^2$
which holds true by Wilmshurst's theorem along with the examples given in \cite{W2} showing that this bound is sharp (shown independently in \cite{BHS1995}).
For $m=1$, the upper bound was proved by D. Khavinson and G. Swiatek \cite{KhSw} using holomorphic dynamics.
A proof of the Crofoot-Sarason conjecture using further tools from holomorphic dynamics \cite{G2008} established that this bound is sharp (which had been verified previously for finitely many values of $n$ in \cite{BL2004}).  Together with the Khavinson-Swiatek result, this verified Wilmshurst's conjecture for $m=1$.
In a related direction, it was shown in \cite{SeteZur} that each valence $k$ can be realized by some (possibly high degree) choice of $p$ and $q$.

\begin{remark}
The above-mentioned use of holomorphic dynamics as an indirect technique, introduced in \cite{KhSw}, was further developed and led to elegant solutions to a variety of problems in potential theory, complex analysis, and mathematical physics 
\cite{KN},
\cite{KhLu2010}, \cite{LeeMakarov}, \cite{BergErem},
\cite{BergErem2018}.  In particular, Khavinson and Neumann used this technique to study the valence of harmonic rational functions, resolving a conjecture in astronomy \cite{Rhie} on the maximum number of images of a single background source that can be produced by a gravitational lens \cite{KN} (the gravitational lensing problem was investigated further in \cite{Bleher}, \cite{SeteCMFT}, \cite{SeteGrav}, \cite{SetePert}, \cite{Zur2018a}, \cite{Zur2018b}).  Lee and Makarov used the technique, along with quasiconformal surgery, to provide sharp estimates for the topology of quadrature domains \cite{LeeMakarov}.  Bergweiler and Eremenko used the technique along with elliptic function theory to give a greatly simplified proof \cite{BergErem} of Lin and Wang's classification of the number of critical points of Green's function on a torus \cite{LinWang}.
\end{remark}

After Wilmshurst's conjecture had been confirmed for the top ($m=n-1$) and bottom ($m=1$) cases, it came as somewhat of a surprise when counterexamples were discovered in \cite{LLL}.
For $m=n-3$, Wilmshurst's conjecture states $N(n,m) = 3n-2 + m(m-1) = n^2 - 4n + 10$,
and counterexamples were provided in \cite{LLL} having valence greater than $n^2 - 3 n + O(1)$ as $n \rightarrow \infty$.
Further counterexamples were provided in \cite{LS} for $m=n-2$ again having valence that exceeds the conjecture by a margin that increases linearly in $n$.
Additional counterexamples were provided in \cite{HLLM} for certain specific choices of $m,n$ using computer-assisted proof (numerical root finding with a posteriori validation).

In spite of these counterexamples, it still seems likely that the main spirit of Wilmhurst's conjecture is correct, and $N(m,n)$ increases at most linearly in $n$ for each $m$ fixed.
However, it was suggested in \cite{LLL} that if such a linear-in-$n$ estimate does hold, then contrary to the form of Wilmshurst's conjecture we may expect the growth rate to depend on $m$.  
The goal of the current paper is to confirm this by providing new counterexamples that show
\be\label{eq:limsup}
\limsup_{m \rightarrow \infty} \limsup_{n \rightarrow \infty} \frac{N(n,m)}{n} = \infty,
\ee
as opposed to Wilmshurst's conjecture where $\limsup_{m \rightarrow \infty} \limsup_{n \rightarrow \infty} \frac{N(n,m)}{n} = 3$ has a finite value.

The statement \eqref{eq:limsup} follows from the following more precise result.

\begin{thm}\label{thm:main}
For each pair of integers $n>m \geq 1$ there exist complex polynomials $p(z),q(z)$ with $p$ of degree $n$ and $q$ of degree $m$ such that the valence of $p(z) + \ol{q(z)}$ is at least $\lceil n\sqrt{m} \rceil$.
\end{thm}

\begin{remark}
The quantity $\lceil n\sqrt{m} \rceil$ exceeds Wilmshurst's conjecture  for each fixed $m>9$ for all $n$ sufficiently large, and more generally when $m$ grows slowly with $n$, namely, $m = O(n^\alpha)$ for $0 \leq \alpha <1/2$.
For all $n>m$, it was shown in \cite{KhavLeeSaez} that there are examples with at least $m^2 + n + m$ many zeros; while this number does not exceed Wilmshurst's conjecture, it shows that there is substantial room for improvement in Theorem \ref{thm:main} in the range where $m$ grows more quickly than $n^{1/2}$.  It seems reasonable to suspect the existence of examples with $n\cdot m$ many zeros.
\end{remark}

The proof of Theorem \ref{thm:main} is probabilistic and rests on an estimate for the expectation of the number of zeros for a carefully chosen model of random harmonic polynomials.
We sketch the main ideas of the proof here.  In the model of random polynomials we will consider, the polynomial $q$ is sampled from the so-called (complex) Kostlan ensemble, a popular Gaussian model of random polynomials.  We then take $p$ to be a (deterministic) perturbation of $q$, namely, $p(z) = \e z^n + q(z)$.
With this choice of $p$, the imaginary part of $h(z) = p(z) + \ol{q(z)}$ is $\e \Im \{z^n \}$ which vanishes on $n$ equally spaced lines through the origin.  The restriction of the real part of $h(z)$ to any one of these lines is distributed as a perturbation of a random polynomial of degree $m$ sampled from the \emph{real} Kostlan ensemble, and its average number of zeros can be estimated using a stability argument when the size $\e>0$ of the perturbation is small. 
Using linearity of expectation, the sum of these individual estimates gives a lower bound for the expected total number of zeros of $h$.  The statement in the theorem will then follow from the simple general fact that there must exist at least one point in parameter space achieving the average. 
See Section \ref{sec:main} for details of the proof (and Section \ref{sec:Kostlan} for preliminary results on the real and complex Kostlan ensembles).

\begin{remark}
The term ``probabilistic lens'' refers to the application of indirect probabilistic techniques for proving deterministic results (this terminology is especially used in Combinatorics where the method is well-established \cite{Alon}).
The current paper is the first instance of this method in the study of the valence of harmonic polynomials;
while there have been a number of papers investigating the average valence of random harmonic polynomials (collectively spanning a variety of models) \cite{LiWei}, \cite{Lerariotruncated},
\cite{Andy}, \cite{AndyZach},
the current paper is the first of these that leads to a new deterministic result.\footnote{Of course, in \emph{other} areas of harmonic function theory, there have already been several well-known indirect probabilistic proofs of deterministic results, often relying on potential theoretic properties of Brownian motion, see \cite{Bass}.} %of the probabilistic lens in the study of the valence of harmonic polynomials.   Those results consider a variety of models, but  the average number of zeros for a variety of models of random harmonic polynomials. those results take a probabilistic perspective for its intrinsic did not lead to any new deterministic statements, and 
\end{remark}

We prove Theorem \ref{thm:main} in Section \ref{sec:main} after we
briefly review the construction and properties of the real and complex Kostlan ensembles in Section \ref{sec:Kostlan}.
We end the paper with some concluding remarks in Section \ref{sec:concl}.

\noindent {\bf Acknowledgments.}
The author thanks the anonymous referee for a careful reading and many helpful comments and corrections that improved the exposition and the clarity of the proofs.

\section{The real and complex Kostlan ensembles}\label{sec:Kostlan}

In this section we review the basic properties of the real and complex Kostlan ensembles of random polynomials.

The (real) Kostlan ensemble can be described by sampling a random polynomial $f$ with independent Gaussian coefficients $\alpha_k$
\be\label{eq:KostlanReal}
f(x)=\sum_{k=0}^m \alpha_k x^k, \quad \alpha_k \sim  N \left(0,\binom{m}{k}\right).
\ee
where the notation
$\alpha_k \sim  N \left(0,\binom{m}{k}\right)$ means that $\alpha_k$ is a (real) Gaussian random variable with mean zero and variance $\binom{m}{k}$.

We will need the following result on the average number of zeros of a real Kostlan polynomial $f$ over an interval of the real line.
For an elegant geometric proof of this result, we refer the reader to \cite[Sec. 3.1.2]{EdelmanKostlan95}.

\begin{lemma}\label{lemma:EK}
Let $f$ be a random real polynomial of degree $m$ sampled from the Kostlan ensemble.
Let $N_f(a,b)$ denote the average number of zeros of $f$ over the interval $(a,b)$.  Then the expectation of $N_f(a,b)$ is given by
\begin{align*}
\EE N_f(a,b) &= \frac{\sqrt{m}}{\pi} \int_a^b \frac{1}{1+t^2} dt \\ 
&= \frac{\sqrt{m}}{\pi} \left[ \arctan(b) - \arctan(a) \right].
\end{align*}
\end{lemma}

The complex Kostlan ensemble can be described similarly where the coefficients $c_k$ are now independent \emph{complex} Gaussians 
\be\label{eq:KostlanComplex}
q(z)=\sum_{k=0}^m c_k z^k, \quad c_k \sim N_\CC\left(0,\binom{m}{k}\right),
\ee
where the notation
$c_k \sim  N_\CC \left(0,\binom{m}{k}\right)$ means that $c_k$ is a complex Gaussian random variable with mean zero and variance $\binom{m}{k}$.
Recall that the complex Gaussian $c_k \sim N_\CC \left(0,\binom{m}{k}\right)$ can be expressed as a sum 
\be\label{eq:ck}
c_k = a_k + i b_k
\ee
where $a_k,b_k \sim N(0,\frac{1}{2} \binom{m}{k})$ are independent real Gaussians.

\begin{remark}
While the binomial coefficient variances appearing in \eqref{eq:KostlanReal} and \eqref{eq:KostlanComplex} might initially seem strange, the Kostlan model is often considered the most natural model from the perspective of the following higher ground. Each Gaussian model of random polynomials is determined by specifying a choice of inner product on the vector space of polynomials up to degree $n$; the Kostlan ensemble corresponds to the Gaussian measure induced by the Bombieri product (a.k.a. the Fischer product),
which is just the $L^2$ inner product associated to integration with respect to the natural Fubini-Study metric on projective space.
As a consequence of the unitary-invariance of the Fubini-Study metric, the Bombieri product, and the zero set of $f$ as well, is invariant under unitary transformations of projective space.  Moreover, the complex Kostlan ensemble is the \emph{unique} (up to multiplcation by a scalar) unitarily-invariant Gaussian ensemble of random complex polynomials \cite[Ch. 12]{BCSS}.
For this reason it is sometimes natural, although we will not do it in this paper, to switch to projective coordinates while working with the Kostlan ensemble.
\end{remark}

\section{Proof of Theorem \ref{thm:main}}\label{sec:main}

We fix attention on a special class of harmonic polynomials $p(z) + \ol{q(z)}$ where $p(z)$ is a perturbation of $q(z)$ of the form $p(z)=\e z^n+q(z)$, with $\e >0$.
Hence
\be\label{eq:special}
p(z) + \ol{q(z)} = \e z^n + q(z) + \ol{q(z)} = \e z^n + 2 \Re q(z).
\ee
We randomize the above harmonic polynomial \eqref{eq:special} by taking $q(z)$ as in \eqref{eq:KostlanComplex} to be randomly sampled from the complex Kostlan ensemble  of random polynomials.

Taking real and imaginary parts of the equation $p(z) + \ol{q(z)}=0$ then gives the following system of equations.
\be\label{eq:system}
\begin{cases}
\e \Re\{ z^n \} + 2\Re \{q(z)\} = 0, \\
\Im \{ z^n \}  = 0.
\end{cases}
\ee

The solution set of the second equation in \eqref{eq:system} consists of $n$ equally spaced lines through the origin, and hence the solutions of the system \eqref{eq:system} correspond to the zeros of the univariate polynomials obtained from restricting the first equation to each of these lines.

In order to restrict the equation $\e \Re\{ z^n \} + 2\Re \{q(z)\} = 0$ to each of these $n$ lines, we write $z=r e^{i \theta_j}$ with $r \in \RR$ and $\theta_j = \frac{j \pi}{n}$ for $j=0,1,2,...,n-1$.  This gives for each $j=0,1,2,...,n-1$
\be\label{eq:polar}
\e (-1)^n r^j + 2 \sum_{k=0}^m \left[ a_k \cos(k \theta_j) - b_k \sin(k\theta_j) \right] r^k = 0,
\ee
where $a_k, b_k$ are the independent real Gaussians appearing in the real and imaginary parts of the coefficients of $q$ as in \eqref{eq:ck}.
Let us multiply \eqref{eq:polar}  by $(-1)^j$ and abbreviate the resulting equations as
\be\label{eq:abbrv}
\e r^n + f_j(r) = 0, \quad j = 0,1,2,...,n-1,
\ee
where 
\be\label{eq:withcoeff}
f_j(r) = 2\sum_{k=0}^m \alpha_{k,j} r^k, \quad \alpha_{k,j}:= (-1)^j \left[ a_k \cos(k \theta_j) - b_k \sin(k\theta_j) \right].
\ee
Let $N_j$ denote the number of real solutions of \eqref{eq:polar}.
Since the event $p(0)+\ol{q(0)} = 0$ has zero probability, we have that the total number $N$ of solutions of the system \eqref{eq:system} is given by 
\be\label{eq:Nsum}
N = \sum_{j=0}^{n-1} N_j,
\ee 
almost surely.

The following lemma will play a key role in estimating the expectation of each $N_j$.

\begin{lemma}\label{lemma:restr}
For each $j$, the polynomial
$f_j$ appearing in \eqref{eq:abbrv} and defined in \eqref{eq:withcoeff} is distributed as (a constant multiple of) a real Kostlan polynomial of degree $m$.
\end{lemma}

\begin{proof}[Proof of Lemma \ref{lemma:restr}]
Fixing $j$, the random coefficients $\alpha_{k,j}$ in \eqref{eq:withcoeff} are independent of one another as $k$ varies. Indeed, this follows from the independence of $a_k,b_k$.
Further using the independence of $a_k,b_k$ along with the fact that they are centered Gaussians with variance $\frac{1}{2}\binom{m}{k}$, it follows from the sum law for Gaussians that
$\alpha_{k,j} = (-1)^j \left[ a_k \cos(k \theta_j) - b_k \sin(k\theta_j) \right] $ is distributed as a centered Gaussian with variance $\frac{1}{2}\binom{m}{k}\cos^2(k \theta_j) + \frac{1}{2} \binom{m}{k} \sin^2(k \theta_j) = \frac{1}{2} \binom{m}{k}$, i.e., we have $\alpha_{k,j} \sim N(0,\frac{1}{2}\binom{m}{k})$.  Equivalently, $\sqrt{2} \, \alpha_{k,j} \sim N(0,\binom{m}{k})$.
We conclude that $(\sqrt{2})^{-1}f_j$ is distributed as a real Kostlan polynomial as desired.
\end{proof}

In view of \eqref{eq:Nsum}, our goal can be reduced to proving the following lemma.

\begin{lemma}\label{lemma:expected}
Given arbitrary $\e_0 > 0$, the expectated number of solutions of \eqref{eq:polar} satisfies 
$$\EE N_j \geq \sqrt{m} - \e_0,$$
for any sufficiently small size $\e>0$ of the perturbation in \eqref{eq:polar}.
\end{lemma}

Before proving the lemma, let us see how it is used to complete the proof of the theorem.
By linearity of expectation and \eqref{eq:Nsum}, we have that the expected number of zeros of $p(z) + \ol{q(z)}$ is given by the sum over $j$ of the expected number of solutions of each equation \eqref{eq:polar}, i.e.,
$$\EE N = \sum_{j=0}^{n-1} \EE N_j.$$
Together with Lemma \ref{lemma:expected} this gives
\be
\EE N \geq n (\sqrt{m} -  \e_0).
\ee
Since the valence of $p(z) + \ol{q(z)}$ is an integer, this estimate for the expectation implies the existence of examples with valence
$\lceil n (\sqrt{m}-\e_0) \, \rceil$.
Finally, we notice that $\lceil n (\sqrt{m}- \e_0) \, \rceil$ equals $\lceil n \sqrt{m} \, \rceil$ when $\e_0>0$ is sufficiently small, as the reader can easily verify while separating the case when $\sqrt{m}$ is an integer.

It remains to prove Lemma \ref{lemma:expected}.

We will use the following quantitative version of Bulinskaya's Lemma which follows from a more general result on Gaussian random fields \cite[Lemma 7]{NazarovSodin2}.
\begin{lemma}
\label{lemma:Bulinskaya}
Let $f$ be a random polynomial sampled from the real Kostlan ensemble.
Given $R>0$, $\delta>0$ there exists $\tau>0$ such that
\be
\PP \left\{ \min_{[-R,R]} \max \{|f(x)|,|f'(x)| \} < \tau \right\} < \delta .
\ee
\end{lemma}

We will also need the following elementary deterministic result.

\begin{lemma}\label{lemma:deterministic}
Given $R_0>0$ and letting $R=R_0+1$,
suppose $f:\RR \rightarrow \RR$ is a $C^1$-smooth function satisfying
\be\label{eq:stable}
\min_{[-R,R]} \max \{|f(x)|,|f'(x)| \} \geq \tau,
\ee
for some $\tau>0$.
Suppose further that $g:\RR \rightarrow \RR$ is a continuous function satisfying
\be 
\sup_{[-R,R]} |g(x)| \leq \tau/2.
\ee
Then we have 
\be
N_{f+g}(-R,R) \geq N_{f}(-R_0,R_0).
\ee
\end{lemma}

\begin{remark}
The lower bound in Lemma \ref{lemma:deterministic} suits our purposes, but a tighter result holds for $C^1$-small perturbations;
the same conditions \eqref{eq:conditions} on $f$ additionally guarantee an upper bound on the number of zeros of $f+g$ in $(-R_0,R_0)$ under the assumption that $g$ is $C^1$-smooth and sufficiently $C^1$-small (this is an instance of the \emph{transversality and stability principle} from differential topology).
\end{remark}

\begin{proof}[Proof of Lemma \ref{lemma:deterministic}]
Suppose \eqref{eq:stable} holds, and let $x_1,x_2,...,x_\ell$ denote the zeros of $f$ in the interval $(-R_0,R_0)$.  Notice that $f$ has finitely many such zeros.  Otherwise, there is an accumulation point $x_0$ of zeros in the interval $[-R_0,R_0]$ where we have $f(x_0)=0$ by continuity and $f'(x_0)=0$ since $f$ is $C^1$-smooth, contradicting the condition \eqref{eq:stable}.

The set  $S_\tau := \{ x\in (-R,R) :  |f(x)|<\tau \}$ is a union of disjoint open intervals $I_k=(a_k,b_k)$.
The condition \eqref{eq:stable} 
guarantees that  $|f'(x)| \geq \tau$ on $S_\tau$,
and by continuity of $f'$ we have on each interval $I_k$ either $f'(x) \geq \tau$ or $f'(x) \leq -\tau$.  In particular, $f$ is strictly monotone on each interval and has at most one zero in $I_k$.
Each of the zeros $x_1,x_2,...,x_\ell$ is contained in $S_\tau$ and hence is in one of the intervals $I_k$.
 Relabeling indices if necessary, we let $I_k=(a_k,b_k)$ denote the interval containing the zero $x_k$ for each $k=1,2,...,\ell$.

Since $|f'(x)| \geq \tau$ on $I_k$, the mean value theorem implies that $b_k \leq x_k+1$ and $a_k \geq x_k-1$, which implies $-R=-R_0-1 < a_k < b_k < R_0+1=R$ for each $k=1,2,...,\ell$.
This means the endpoints of each interval are contained in $(-R,R)$.

We have the following conditions at the endpoints of each interval
\be\label{eq:conditions}
\begin{cases}
|f(a_k)|=|f(b_k)|=\tau, \\
f(b_k)=-f(a_k)
\end{cases}.
\ee
Indeed, the first condition
$|f(a_k)|=|f(b_k)|=\tau$ follows from the fact that $a_k,b_k$ are in $(-R,R) \setminus S_\tau$ along with the continuity of $f$ and the definition of $S_\tau$.
The second condition $f(b_k) = - f(a_k)$ then follows from the strict monotonicity of $f$ over each interval $I_k$.

By the intermediate value property, the conditions \eqref{eq:conditions}
imply that for any continuous function $g$ satisfying $\sup_{[-R,R]} |g(x)| \leq \tau/2$, we have $f+g$ has a zero in each of the intervals $I_k$.
Recalling that the intervals $I_k$ are disjoint and contained in $(-R,R)$, this implies that $N_{f+g}(-R,R) \geq N_{f}(-R_0,R_0)$ and completes the proof of the lemma.
\end{proof}

\begin{proof}[Proof of Lemma \ref{lemma:expected}]
Let $\e_1>0$ and $\delta>0$ be arbitrary.
By Lemma \ref{lemma:restr} and Lemma \ref{lemma:EK}, we have
$$\EE N_{f_j}(a,b) = \sqrt{m}\frac{\left(\arctan(b)-\arctan(a) \right)}{\pi},$$
where $N_{f_j}(a,b)$ denotes the number of zeros of $f_j$ over the interval $(a,b)$.
Hence,  for all $R_0>0$ sufficiently large we have
\be\label{eq:KRest}
\EE N_{f_j}(-R_0,R_0) \geq \sqrt{m} - \e_1.
\ee

Let $R=R_0+1$, and let $A$ denote the event that the condition \eqref{eq:stable} holds with $f_j$ playing the role of $f$.

Applying Lemma \ref{lemma:Bulinskaya}, we find $\tau>0$ such that the probability of the event $A$ satisfies
\be\label{eq:probA}
\PP (A) \geq 1 - \delta.
\ee

The supremum over $r \in [-R,R]$ of the perturbation appearing in \eqref{eq:abbrv} satisfies 
$$ \sup_{[-R,R]} | \e  r^n | \leq \tau/2$$ 
for all $\e>0$ sufficiently small.
With such choice of $\e$, if the event $A$ occurs, then we can apply Lemma \ref{lemma:deterministic} (with $f_j$ playing the role of $f$ and $\e r^n $ playing the role of $g$) to conclude 
$N_j \geq N_{f_j}(-R_0,R_0)$.
Hence, conditioning on the event $A$ and taking expectations, we have
\be\label{eq:conditional}
\underset{A}{\EE} N_j \geq \underset{A}{\EE} N_{f_j}(-R_0,R_0),
\ee
where $\underset{A}{\EE}$ denotes conditional expectation.

We have
\be
\EE N_{f_j}(-R_0,R_0) = \PP(A) \underset{A}{\EE} N_{f_j}(-R_0,R_0) + (1-\PP(A)) \underset{A^c}{\EE} N_{f_j}(-R_0,R_0),
\ee
which implies
\begin{align*}
    \underset{A}{\EE} N_{f_j}(-R_0,R_0) &= \frac{1}{\PP(A)} \left( \EE N_{f_j}(-R_0,R_0) - (1-\PP(A)) \underset{A^c}{\EE} N_{f_j}(-R_0,R_0) \right) \\
\text{(by \eqref{eq:probA})} \quad     &\geq \EE N_{f_j}(-R_0,R_0) - \delta \underset{A^c}{\EE} N_{f_j}(-R_0,R_0) \\
\quad      &\geq \EE N_{f_j}(-R_0,R_0) - \delta m,
\end{align*}
where in the final line above we have used that the number of real zeros of a degree-$m$ polynomial is at most $m$.
Together with \eqref{eq:KRest} the above inequalities imply
\be\label{eq:N_jA}
\underset{A}{\EE} N_{f_j}(-R_0,R_0) \geq \sqrt{m} - \e_1 - \delta m.
\ee

We have 
\begin{align*}
    \EE N_j &= \PP(A) \underset{A}{\EE} N_j + (1-\PP(A))\underset{A^c}{\EE} N_j \\
    &\geq \PP(A) \underset{A}{\EE} N_j \\
\text{(by \eqref{eq:probA})} \quad      &\geq (1-\delta) \underset{A}{\EE} N_j.
\end{align*}
Together with \eqref{eq:conditional} and \eqref{eq:N_jA} this implies
\be
\EE N_j \geq (1-\delta) \left(\sqrt{m} - \e_1 - \delta m \right).
\ee
Since $\delta>0$ and $\e_1>0$ were arbitrary, it follows that for any $\e_0>0$ we have $\EE N_j \geq \sqrt{m} - \e_0$ for any sufficiently small size $\e>0$ of the perturbation appearing in \eqref{eq:polar}.  This verifies the lemma and hence completes the proof of the theorem.
\end{proof}

\section{Concluding remarks}\label{sec:concl}

\subsection{Upper bound for the number of zeros of $\e z^n + q(z) + \ol{q(z)}$} After the initial discovery of counterexamples to Wilmshurst's conjecture, it was conjectured in \cite[Conj. 1.4]{LLL} that $N(n,m) \leq 2m(n-1)+n$.  So far, there has been no progress on this conjecture, but we note that it is easy to verify for the special class of polynomials we considered in the proof of Theorem \ref{thm:main}, namely polynomials of the form $\e z^n + q(z) + \ol{q(z)}$.
Indeed, this follows from an application of Descartes' rule of signs as we explain below.
We may assume $m \geq 2$ since the case $m=1$ follows from the result of Khavinson and Swiatek \cite{KN} mentioned in the introduction.
Additionally, we may assume $n \geq 4$ since the case $n=3$ and $m=2$ (this is the only case to consider with $n<4$ since we are assuming $n>m \geq 2$) follows from Wilmshurst's theorem (the Bezout bound).
Apply Descartes' rule of signs to the real polynomials $\e r^n + f_j(r)$ obtained as before by restricting $\e z^n + 2 \Re \{ q(z) \}$ to each of the $n$ lines $\{\Im z^n = 0 \}$.  Fix $j$ and let $s_+$ and $s_-$ denote the number of sign changes of coefficients of $f_j(r)$ and $f_j(-r)$ respectively. Let $s_0$ denote the multiplicity of $0$ as a root, i.e., $s_0$ is the minimal degree of the monomials appearing in $f_j(r)$ (we will have $s_0=0$ in the generic case that $0$ is not a root). Since $f_j$ is a degree-$m$ polynomial, we have $s_+ + s_- + s_0 \leq m$ (this is an elementary fact which follows from noticing that the only pairs of consecutive monomials that may contribute to both counts $s_+$ and $s_-$ must have a decrement in the exponent of at least two).  Then the number of sign changes of the coefficients of $\e r^n+f_j(r)$ is at most $s_+ + 1$, and the number of sign changes of the coefficients of $\e (-r)^n + f_j(-r)$ is at most $s_- +1$.  The multiplicity $s_0$ of $0$ as a root is unchanged with the additional term $\pm \e r^n$. This yields by Descartes' rule of signs that there are at most $m+2$ many zeros on each of the $n$ lines for a total of at most $n(m+2)$ many zeros.  Recalling that we have reduced to the case $n\geq 4$ and $m \geq 2$, we further notice that $n(m+2)$ is at most $2m(n-1)+n$.
Namely, we have $(m-1)(n-2) \geq 2 $ which implies $m(n-1) \geq m + n$, and adding $m(n-1)+n$ to each side of this inequality, we arrive at $n(m+2) \leq 2m(n-1)+n$ as desired.

\subsection{Derandomize the proof of Theorem \ref{thm:main}}
As with any non-constructive existence proof based on probabilistic methods, a natural open problem stemming from the current work is to choose $q$ strategically, rather than sampling $q$ randomly, in order to produce a large number of zeros of $h(z) = \e z^n + q(z) + \ol{q(z)}$.
In particular, beyond giving a constructive proof of Theorem \ref{thm:main}, it would be interesting to find examples that saturate the above upper bound coming from Descartes' rule of signs, i.e., construct examples of the form $h(z) = \e z^n + q(z) + \ol{q(z)}$ having asymptotically $m \cdot n$ many zeros.
Of course, one can also try to show this while continuing to use a probabilistic approach, namely, by selecting an alternative ensemble of random polynomials from which to sample $q$.

\subsection{Linearity of expectation and the probabilistic lens}\label{sec:linearity}
The main thrust of the proof of Theorem \ref{thm:main} hinges on linearity of expectation.
Restricting $\Re h$ to the various lines through the origin gives rise to \emph{dependent} random polynomials, i.e., the random coefficients $a_{k,j}$ defined in \eqref{eq:withcoeff} are highly dependent as $j$ varies with $k$ fixed.  A pillar of the probabilistic lens is that one can often ignore such dependence thanks to the linearity of expectation (this point is beautifully demonstrated in \cite{Alon} through many combinatorial examples).
In the application presented above, it is remarkable that this technique generates $n\lceil \sqrt{m} \rceil$ many desirable events (each event in this case being the occurrence of a zero) while working within a low-dimensional parameter space---the  polynomials $h$ only have $2m+3$ real parameters.  It seems likely that this technique holds further potential for generating examples in other extremal problems, particularly where the number of parameters is much smaller than the number of instances of desired behavior.

\bibliographystyle{abbrv}
\bibliography{RandHarmonic}

\end{document}